\documentclass[a4paper,12pt,oneside]{amsart}

\usepackage{amssymb}  % defines \nmid, for example
\usepackage{latexsym} % for \Box
\usepackage{comment}
\usepackage{color}
\usepackage{colonequals}

 \usepackage{epsfig}
\usepackage{tikz}

\usepackage{dsfont}

\definecolor{darkgreen}{rgb}{0,0.5,0}
\usepackage[
        colorlinks, citecolor=darkgreen,
        pdfauthor={Roberto Pignatelli},
        pdftitle={Template},
        linktocpage        
]{hyperref}

\usepackage[alphabetic,backrefs,lite]{amsrefs} % for bibliography

\usepackage[all]{xy} % added "lite" to amsrefs and moved xy here; hope this works

\usepackage{enumerate} % for flexible labels in enumerated lists

\newcommand{\CC}{\ensuremath{\mathbb{C}}}

\newcommand{\NN}{\ensuremath{\mathbb{N}}}
\newcommand{\PP}{\ensuremath{\mathbb{P}}}

\newcommand{\ZZ}{\ensuremath{\mathbb{Z}}}

\DeclareMathOperator{\Aut}{Aut}

\newcommand\dual{\mathrel{\raise3pt\hbox{$\underline{\mathrm{\thinspace d
\thinspace}}$}}}
\newcommand\qe{\ifhmode\unskip\nobreak\fi\quad $\Box$}       % box for QED

\def\BOX{\hfill\lower.5\baselineskip\hbox{$\Box$}}

\newtheorem{theorem}{Theorem}[section]
\newtheorem{lemma}[theorem]{Lemma}
\newtheorem{corollary}[theorem]{Corollary}
\newtheorem{proposition}[theorem]{Proposition}

\newtheorem*{theorem*}{Theorem}
\newtheorem*{problem*}{Problem}
\newtheorem*{question*}{Question}

\theoremstyle{remark}
\newtheorem{remark}[theorem]{Remark}

\theoremstyle{definition}
\newtheorem{definition}[theorem]{Definition}

\numberwithin{equation}{section}

\newcounter{nootje}
\renewcommand\check[1]
  {\marginpar{\tiny\begin{minipage}{20mm}\begin{flushleft}\thenootje : #1\end{flushleft}\end{minipage}}\addtocounter{nootje}{1}}
\begin{document}

\title{Some surfaces with canonical map of degree $4$.}

\author{Federico Fallucca}
\address{Dipartimento di Matematica,
	Universit\`a di Trento,
	via Sommarive 14,
	I-38123 Trento, Italy.}
\email{Federico.Fallucca@unitn.it}

\author{Roberto Pignatelli}
\address{Dipartimento di Matematica,
	Universit\`a di Trento,
	via Sommarive 14,
	I-38123 Trento, Italy.}
\email{Roberto.Pignatelli@unitn.it}
\date{\today}
\thanks{
This paper originated by several enlighting discussions of the second author with  C. Glei\ss ner and C. Rito on the idea of considering product-quotient surfaces $C_1 \times C_2/G$ with a subgroup $H \subset G$ such that the canonical map of  $C_1 \times C_2/H$ factors through  $C_1 \times C_2/G$. He thanks both of them heartily.
We thank J. Stevens, S. Troncoso and the anonymous user ``mathlove'' on Stack Exchange for interesting discussions on the behaviour of the function $\sigma$, and N. Bin for pointing out a useful reference.
The second author is indebted with M. Mendes Lopes and R. Pardini for sharing a preliminary version of their inspiring survey. The second author thanks M. Penegini and F. Polizzi for inviting him to give a talk at the $8^{th}$ ECM in Portorose, where he presented a preliminary version of this result: the beautiful atmosphere and the clever questions posed there helped us in  improving our results. 
}

 \makeatletter
    \def\@pnumwidth{2em}
  \makeatother

\begin{abstract}
	In this note we construct unbounded families of minimal surfaces of general type whose canonical map has a degree of $4$ such that the limits of the slopes $K^2/\chi$ assume countably many different values in the closed interval $\left[6+\frac23 ,8\right]$.
\end{abstract}

\maketitle

\addtocontents{toc}{\protect\setcounter{tocdepth}{1}}

%\tableofcontents
%=========================================================================

%%%%%%%%%%%%%%%%%%%%%%%%%%%%%%%%%%%%%%%%%%%%%%%%%%%%%%%%%%%%%%%%%%%%%%%
\section*{Introduction}

In this paper a surface is a compact complex manifold of dimension $2$. An {\it unbounded family} of surfaces is a sequence of surfaces $S_n$ with an arbitrarily large Euler characteristic $\chi({\mathcal O}_{S_n})$. More precisely, our unbounded families are sequence of surfaces $S_n$ such that $\lim_{n \to \infty} \chi({\mathcal O}_{S_n})=+\infty$.

It is well known since the pioneering  work of Beauville \cite{Beauville} and a Theorem of Xiao Gang \cite{Xiao Gang} that the degree of the canonical map of a surface $S$, if we assume a large enough Euler characteristic, is bounded from above by $8$. Recall that the degree of the canonical map is a birational invariant, so we can without loss of generality assume that $S$ is minimal. 

We address the reader to the beautiful survey  of M. Mendes Lopes and R. Pardini \cite{survey} on the subject. 
We read from there, among  other things, examples of unbounded sequences of minimal surfaces whose canonical map has a degree of $\delta$ for every $\delta \in \{2,4,6,8\}$. 

Recall that the slope $\mu$ of a minimal surface $S$ is defined as $\mu(S):=\frac{K^2_{S}}{\chi({\mathcal O}_{S})}$.
By the Bogomolov-Miyaoka-Yau inequality  $\mu(S) \leq 9$. By the above mentioned results it easily follows that for any unbounded family $S_n$ of minimal surfaces whose canonical map has a degree of $\delta$, $\liminf \mu(S_n) \geq \delta$.
This raises the question of investigating, for all $\delta$, the set of the accumulation points of the slopes of unbounded families of minimal surfaces whose canonical map has a degree of $\delta$.
Compare \cite{survey}*{Question 5.6}.

We know only three constructions of unbounded families of minimal surfaces whose canonical map has a degree of $4$. 

The first, mentioned in \cite{survey}, is obtained by taking the product of two hyperelliptic curves. All these surfaces have slope $8$.

The second, see \cite{Bin}*{Remark 3},  is a construction as Galois cover of $\PP^1 \times \PP^1$ with Galois group  $(\ZZ/2\ZZ)^3$; they also have $\lim \mu\left( S_n \right)$ equal to $8$.

The last, constructed by F. J. Gallego and G. P. Purnaprajna, give unbounded families with $\lim \mu\left( S_n \right)$ equal either to $8$ or to $4$, see the last column of  \cite{GP1}*{Table at page 5491}.

Inspired by certain constructions of $K3$ surfaces in \cite{AliceMatteo}, we show that $\lim \mu\left( S_n \right)$, when $S_n$ is an unbounded families of minimal surfaces whose canonical map has a degree of $4$, may assume infinitely many different values. More precisely 
\begin{theorem*}
There are countably many unbounded sequences $S_n$ of surfaces of general type that have canonical map of degree $4$
such that
$
\lim_{n\rightarrow \infty} \mu \left( S_n\right)
$
assumes pairwise distinct values in the range $\left[6 + \frac23 ,8\right]$. 
\end{theorem*}

  All these surfaces are product-quotient surfaces. The product-quotient surfaces have been introduced by the second author, I. Bauer, F. Catanese and F. Grunewald in \cite{4names}
 (compare also \cite{fibred}, \cite{Fano}, \cite{orbifold}). 
 Their canonical map was studied, in the special case of the surfaces isogenous to a product, in \cite{canonicalisogenous}. 
To our knowledge they were used first for constructing surfaces with canonical map of high degree in \cite{GPR}. 
%We plan to use them to construct more examples in the future.

\subsection*{Notation}

For each real number $z$, let $\lceil z \rceil$ be the smallest integer greater or equal than $z$.

For each pair of integers $z,n \in {\mathbb N}$ we denote by $[z]_n$ the unique integer, $0\leq [z]_n \leq n-1$, such that $z-[z]_n$ is divisible by $n$.

We say that a point of a complex analytic variety  is a {\it singular point of type $\frac{p}{q}$}, with $p\in {\ZZ}\setminus \{0\}$, $q \in \NN \setminus \{0\}$, $\gcd(p,q)=1$,   if one of its  neighbourhoods is analytically isomorphic to the quotient of a neighbourhood of the origin of $\CC^2$ by the cyclic group generated by the automorphism $(x,y) \mapsto (e^{\frac{2\pi i}{q}}x,e^{p\frac{2\pi i}{q}}y)$. These singularities are most commonly denoted in the literature as {\it cyclic quotient singularities of type $\frac1{q}(1,r)$}, where $r$ is the remainder of the division of $p$ by $q$.

We say that a variety has {\it basket} of singularities $a_1 \frac{p_1}{q_1} + a_2 \frac{p_2}{q_2}+\cdots + a_r \frac{p_r}{q_r} $ if its singular locus is finite and can be partitioned in $r$ subsets $S_1,\ldots,S_r$ of respective cardinality $a_1,\ldots,a_r$ such that each point in $S_j$ is a singularity of type $\frac{p_j}{q_j}$.

%%%%%%%%%%%%%%%%%%%%%%%%%%%%%%%%%%%%%%%%%%%%%%%%%%%%%%%%%%%%%%%%%%%%%%%

\section{Generalized Wiman Curves}

By a  classical result of Harvey and Wiman (\cites{harvey, wiman}) an automorphism of a curve of genus $g$ at least $2$ has order at most $4g+2$. Moreover, there is exactly one curve of genus $g$ with an automorphism of order $4g+2$ for each integer $g \geq 2$, usually referred in literature as the {\it Wiman curve of genus $g$}. 

\begin{definition}[\bf Generalized Wiman curves]\label{def_WimanCurves}

Consider two positive integers $n,d \geq 1$.

A generalized Wiman curve of type $n,d$ is a curve in the weighted projective space $ \PP\left( 1,1,\left\lceil \frac{nd}2 \right\rceil \right)$ defined by an equation of the form 
\[
y^2=x_0^{[nd]_2} f(x_0^n,x_1^n) 
\] 
where $f$ is a homogenous polynomial  of degree $d$ in the two variables $x_0,x_1$  without multiple roots such that neither $x_0$ nor $x_1$ divide $f$.

\end{definition}

\begin{remark} The assumptions on the polynomial $f$ ensure that any generalized Wiman curve is smooth.
\end{remark}

By adjunction a generalized Wiman curve $C$ of type $n,d$ has genus $g=\left\lceil \frac{nd}2 \right\rceil-1$. In fact a basis of $H^0(C,K_C)$ is given by the monomials
\begin{equation}\label{monomials}
x_0^{\left\lceil \frac{nd}2 \right\rceil-2}, x_0^{\left\lceil \frac{nd}2 \right\rceil-3}x_1, \ldots , x_0x_1^{\left\lceil \frac{nd}2 \right\rceil-3}, x_1^{\left\lceil \frac{nd}2 \right\rceil-2} 
\end{equation}

A generalized Wiman curve of type $n,d$ has  the following two natural commuting automorphisms 
\begin{align*}
\iota \colon& (x_0,x_1,y) \mapsto (x_0,x_1,-y)&
\rho \colon& (x_0,x_1,y) \mapsto (x_0,e^\frac{2 \pi i}{n}x_1,y)&
\end{align*}
of respective order $2$ and $n$. This shows
\begin{enumerate}
\item all generalized Wiman curves are hyperelliptic, $\iota$ being their hyperelliptic involution;
\item a generalized Wiman curve of type $2g+1,1$ is the Wiman curve of genus $g$.
\end{enumerate} 

Since $\iota$ is the hyperelliptic involution, $\iota$ acts on $H^0(C,K_C)$ as the multiplication by $-1$.
The points fixed by $\iota$ are the $2g+2$ points of the divisor $y=0$. 

\begin{definition}
We will say that $\rho$ is the {\it rotation} of $C$.
\end{definition}

We conclude this section by studying the action of the rotation.

\begin{proposition}\label{prop_rhoaction}
The action of $\rho$ on the locus $x_0x_1 \neq 0$  has all orbits of order $n$.

The divisor $x_1=0$ is given by two points, both fixed by $\rho$.

If both $n$ and $d$ are odd, then the divisor $x_0=0$ is given by one single point, fixed by $\rho$. Else the divisor $x_0=0$ is given by two distinct points, fixed by $\rho$ if $d$ is even and exchanged by $\rho$  if $d$ is odd.

The monomials in \eqref{monomials} are eigenvalues for the induced action of $\rho$ on $H^0(C,K_C)$.
More precisely $\rho$ acts on them as
\begin{equation}\label{eq_shift}
x_0^{\left\lceil \frac{nd}2 \right\rceil-2-a}x_1^a \mapsto  e^{(a +1) \frac{2\pi i}n} x_0^{\left\lceil \frac{nd}2 \right\rceil-2-a}x_1^a 
\end{equation}
\end{proposition}

\begin{proof}
The rotation lifts the automorphism of $\PP^1=C/\iota$ acting as $(x_0,x_1)\mapsto (x_0,e^\frac{2 \pi i}{n} x_1)$, which fixes only the two points $x_0x_1=0$, so the analogous statement holds for $\rho$.

By the definition of $f$ the point $(x_0,x_1)=(0,1)$ is a branching point of the hyperelliptic $2:1$ map $C \rightarrow \PP^1$ if and only if both $n$ and $d$ are odd, in which case the divisor $x_0=0$ in $C$ is a single (double) point, that is therefore fixed by $\rho$.
Else, if $nd$ is even, $x_0=0$ is formed by two distinct points with homogeneous coordinates $(x_0,x_1,y)=(0,1,\pm \bar{u}_0)$ for some $\bar{u}_0 \neq 0$. 
These two points are either fixed or exchanged by $\rho$.
By the properties of the weighted projective space they are fixed by $\rho$ if and only if $\left(e^{-\frac{2\pi i}{n}}\right)^{ \frac{nd}2 }=1$. We conclude the analysis of the divisor $x_0=0$ by observing that the last equation is verified if and only if $d$ is even.

 Since the point $(x_0,x_1)=(1,0)$ is not a branching point of the hyperelliptic map, the divisor $x_1=0$ is made by two distinct points with coordinates
 $(x_0,x_1,y)=(1,0,\pm \bar{u}_1)$ for some $\bar{u}_1 \neq 0$, both obviously fixed by $\rho$. 
 
 The function $z:=x_1/x_0$ is a local coordinate in both of then, on which $\rho$ acts as $z \mapsto e^\frac{2\pi i}n z$. The adjunction formula maps a monomial $x_0^{\left\lceil \frac{nd}2 \right\rceil-2-a}x_1^a$ to the form that locally restricts to $z^a dz$ and therefore $\rho$ acts on it as the multiplication by $e^{(a +1)\frac{2\pi i}n}$.
 \end{proof}

\iffalse
\begin{corollary}\label{cor_d>3}
If $d\geq 3$ then all $n-$roots of the unity are eigenvalues for the action of the rotation on $H^0(C,K_C)$
\end{corollary}
\begin{proof}
By \eqref{eq_shift} the spectrum of this action is given by $\left\lceil \frac{nd}2 \right\rceil-1$ consecutive $n-$th roots of the unity. This proves the claim since $d\geq 3$ implies $\left\lceil \frac{nd}2 \right\rceil-1 \geq n$.
\end{proof}
\fi

\section{Wiman product-quotient surfaces}

\begin{definition}\label{def_WimanSurfaces}
For all integers $n,d_1,d_2$ and for all $1\leq k \leq n-1$ with $\gcd(k,n)=1$ we define a {\it Wiman product-quotient surface of type $n,d_1,d_2$} with {\it shift} $k$ to be the minimal resolution $S$  of the singularities of its {\it quotient model} $X:=\left( C_1 \times C_2)\right /H$ where 
\begin{itemize}
\item $C_j$, $j=1,2$  is a generalized Wiman curve of type $n,d_j$;
\item $H \subset \Aut (C_1 \times C_2)$ is the cyclic subgroup of order $n$ generated by the automorphism 
\[
(x,y) \mapsto \left( \rho_1 x , \rho_2^k y \right).
\]
\end{itemize}
where $\rho_j$ is the rotation of $C_j$.
\end{definition}

Denote the hyperelliptic involution of $C_j$ by $\iota_j$. Then $\Aut(C_1 \times C_2)$ contains a subgroup of order $4$ generated by $(\iota_1,1)$ and 
$(1, \iota_2)$. The corresponding quotient of $C_1 \times C_2$ is isomorphic to $\PP^1 \times \PP^1$. 
Since this group commutes with $H$ and it intersects  $H$ trivially, it defines a subgroup $K\cong\left( {\mathbb Z}/2{\mathbb Z}\right)^2$ of $\Aut(X)$. Note that $X/K$ is dominated by $\PP^1 \times \PP^1$ and therefore it is rational.

\begin{lemma}\label{lem_factorsK}
The canonical map of $S$ factors through the rational surface $X/K$.
\end{lemma}
\begin{proof}
By the Kuenneth formula 
\[
H^0(C_1 \times C_2, K_{C_1 \times C_2}) \cong H^0(C_1 , K_{C_1})  \otimes H^0(C_2, K_{C_2})
\]
and then both involutions $(\iota_1,1)$ and $(1, \iota_2)$ act on $H^0(C_1 \times C_2, K_{C_1 \times C_2})$  as the multiplication by $-1$.
Since by Freitag Theorem \cite{freitag}*{Satz 1} the pull-back map sends $H^0(S,K_S)=H^0(X,K_X)$ isomorphically onto the invariant subspace $H^0(C_1 \times C_2, K_{C_1 \times C_2})^H$,  it follows that all elements of $K$ act on $H^0(S,K_S)=H^0(X,K_X)$ as a multiple of the identity.

This implies that $H^0(S,K_S)$ cannot separate two points in the same orbit by the action of $K$.
\end{proof}

In the ``degenerate'' case $n=1$, $S=X$ is the product of the two hyperelliptic curves $C_1$ and $C_2$. Assuming 
 $d_1,d_2 \geq 5$ (to have genera at least $2$) we find an unbounded family of surfaces with canonical map of degree $4$ as those mentioned in 
  \cite{survey}.
  
 The degree of the canonical map remains in fact $4$ also for bigger $n$. 
\begin{theorem}\label{thm_deg4}
Let $S$ be a Wiman product-quotient surface of type $n,d_1,d_2$ and assume $n\geq 2$.

\begin{enumerate}
\item If $d_1,d_2 \geq 3$, then $K_S$ is nef.
\item If $d_1 \geq 4$, $d_2 \geq 5$ then the canonical map of $S$ has degree $4$.
\end{enumerate}
\end{theorem}
\begin{proof}
We denote by $x_0,x_1,y$ the coordinates of the weighted projective space containing $C_1$ as in 
Definition \ref{def_WimanCurves}, and by $\bar{x}_0,\bar{x}_1,\bar{y}$ the analogous coordinates for $C_2$. By the Kuenneth formula the monomials
\[
m_{a,b}:=x_0^{\left\lceil \frac{nd_1}2 \right\rceil-2-a} 
\bar{x}_0^{\left\lceil \frac{nd_2}2 \right\rceil-2-b} x_1^a \bar{x}_1^b 
\]
form a basis  of eigenvectors for the action of the $ \left( \rho_1 , \rho_2^k \right)$ of $H$ on $H^0(C_1 \times C_2, K_{C_1 \times C_2})$ with respective eigenvalues
$
 e^{ (a+1+k(b+1)) \frac{2\pi i}n}
$.
So a basis of $H^0(S,K_S)$ is given by the monomials  
\begin{equation}\label{eq_BasisOfKS}
\left\{ m_{a,b} | n \text{ divides } a+1+k(b+1) 
\right\}
\end{equation}

\begin{enumerate}
\item Pulling back $H^0(S,K_S)$ to $C_1 \times C_2$ we obtain a linear system $\Gamma$ defined by the vector subspace $V \subset H^0(C_1 \times C_2,K_{C_1 \times C_2})$ generated by the monomials $m_{a,b}$ in \eqref{eq_BasisOfKS}.

We claim that if both $d_j$ are at least $3$, then the base locus of $\Gamma$ is finite. 

We first note that the divisor defined by each $m_{a,b}$ on $C_1 \times C_2$ is a linear combination of the $4$ divisors $x_0=0$, $\bar{x}_0=0$, $x_1=0$, $\bar{x}_1=0$. Then the base locus of $\Gamma$ is contained in the union of these $4$ divisors. 

We show that the intersection of the base locus of $\Gamma$ with $x_1=0$ is finite. It suffices to prove that there is a monomial in $V$ of the form $m_{0,b}$. In other words, that there is an integer $0\leq b \leq \left\lceil \frac{nd_2}2 \right\rceil-2$ so that $n$ divides $1+k(b+1)$, which is equivalent to ask that the remainder class of $b$ module $n$ is the unique class solving the corresponding congruence. 
Since $d_2\geq 3$, $\left\lceil \frac{nd_2}2 \right\rceil-2 \geq n-1$ and therefore we can find a $b$ in our range for any such a class, giving a monomial $m_{0,b}$ in $V$.

A similar argument show that the intersection of the base locus of $\Gamma$ with each of the other three divisors  $x_0=0$, $\bar{x}_0=0$, $\bar{x}_1=0$ is finite, by showing the existence of a monomial in $V$ of respective type $m_{\left\lceil \frac{nd_1}2 \right\rceil-2,b}$, $m_{a,\left\lceil \frac{nd_2}2 \right\rceil-2}$ and $m_{a,0}$. This concludes the proof of the claim.

Since the base locus of $\Gamma$ is finite, the base locus of $|K_X|$ is finite too whereas the base locus of $|K_S|$ may contain some irreducible curves, all exceptional for the map $S\rightarrow X$, the minimal resolution of the singularities of $X$. In particular there is no $(-1)$-curve in the base locus of $|K_S|$. But a $(-1)$-curve on a surface $S$ is always in the base locus of $|K_S|$! So $S$ is a minimal surface, in the sense that it does not contain $(-1)$-curve.
Since the canonical system is not empty, then $S$ minimal implies that $K_S$ is nef.

\item If $d_1 \geq 4$, $d_2 \geq 5$, arguing as above, we can find a monomial  of the form $m_{0,b}$ in $V$ such that also $m_{0,b+n}$, $m_{n,b}$, $m_{n,b+n}$ belong to $V$. These $4$ monomials  map $C_1\times C_2$  as 
$x_0^n\bar{x}_0^n,x_0^n\bar{x}_1^n, x_1^n\bar{x}_0^n,x_1^n\bar{x}_1^n$
onto a smooth quadric $Q \subset \PP^3$. Then the canonical image of $S$, dominating $Q$, is a surface as well.

Choose a general point $q \in Q$. Its preimage in $C_1 \times C_2$ has cardinality $(2n)^2$, giving $4n$ points of $S$. The group $K$ acts freely on them, giving $n$ smooth points $q_1,\ldots,q_n$ of $X/K$. We know by Lemma \ref{lem_factorsK} that the canonical map of $X$ factors through $X/K$; we finish the proof by showing that it separates the $q_j$.

The automorphism $(\rho_1,1)$ of $C_1 \times C_2$ commutes with $H$, so it defines an automorphism $\rho_X$ of $X$. This automorphism commutes with $K$, so inducing a further automorphism $\rho_K$ of order $n$ of $X/K$. A straightforward direct computation shows that $\rho_K$ permutes the $q_j$ cyclically.  

Now choose a monomial in $V$ of the form  $m_{1,c}$. Then the action of $(\rho_1,1)$ on the vector subspace of $V$ generated by $m_{1,c}, m_{0,b},m_{0,b+n}, m_{n,b}, m_{n,b+n} $ has exactly two distinct eigenvalues, which differ by a primitive $n-$th root of the unity. This implies that  the canonical map of $X$ separates the $q_j$.
\end{enumerate}
\end{proof}

\begin{remark}
The statement of Theorem \ref{thm_deg4} is not meant to be sharp. For example, essentially the same proof shows that part (2) extends to the case $d_1=3$ with the possible exception $n=2$.
\end{remark}
\begin{remark} 
The proof of Theorem \ref{thm_deg4}, part (1) shows that the canonical system of these surfaces has no fixed components.

In fact, it contains all the elements necessary to explicitly compute the base locus of the canonical system, by describing its pull-back on $C_1 \times C_2$, the base locus of the linear system $\Gamma$.

Consider for example the first case $n=2$, $d_1=d_2=3$. In this case $k=1$. Then the given basis of $H^0(S,K_S)$ is $\left\{ x_0\overline{x}_0, x_1\overline{x}_1 \right\}$. This implies that the base locus of $\Gamma$ is formed by $8$ simple points, four defined by $x_0=\overline{x}_1=0$ and four defined by $\overline{x}_0=x_1=0$.
The involution defining $S$ as quotient of $C_1 \times C_2$ acts on these eight points freely, so  $H^0(S,K_S)$ has exactly four simple base points, their images.

By Proposition \ref{prop_rhoaction} this involution fixes exactly $4$ points, those at $x_1=\overline{x}_1=0$, inducing $4$ singular points of type $A_1$ on $S$. The standard formulas from  \cite{orbifold} give  $K^2_S=4$ and $p_g(S)=q(S)=2$,  confirming that the canonical system is a pencil with $4$ base points.
\end{remark}

\section{Unbounded sequences of Wiman Product-Quotient surfaces}

In this section we only consider Wiman product-quotient surfaces of type $n,d_1,d_2$ with both $d_1,d_2$ even.

Identifying a point of $X$ with an orbit of the action of $H$ on $C_1 \times C_2$, the singular points of $X$ correspond to the orbits of cardinality smaller than $n$.

By Proposition \ref{prop_rhoaction} the orbits of the rotation of  a generalized Wiman curve of type $n,d$ with $d$ even are all of order $n$ with $4$ exceptions, $4$ fixed points.
So $X$ has $16$  singular points. A straightforward computation shows that $8$ are of type $\frac{k}{n}$ and $8$ of type $\frac{-k}{n}$.

\iffalse

where the group generated by the rotation does not act freely is made by $3$ or $4$ points, and this implies that the singular points of $X$ are the images of $9$, $12$ or $16$ points of $C_1 \times C_2$. 

More precisely a direct straightforward computation shows that up to exchange $C_1$ and $C_2$ we have the following five cases:
\begin{enumerate}
\item $d_1$ and $d_2$ are even: $X$ has $16$ singular points and basket $8\frac{k}{n}+8\frac{-k}{n}$ 
\item $n$ even, $d_1$ even, $d_2$ odd: $X$ has $12$ singular points and basket  $4\frac{k}{n}+4\frac{-k}{n}+2\frac{k}{n/2}+2\frac{-k}{n/2}$  
\item $n$ odd, $d_1$ even, $d_2$ odd: $X$ has $12$ singular points and basket  $4\frac{k}{n}+4\frac{-k}{n}+2\frac{k(n-1)/2}{n}+2\frac{-k(n-1)/2}{n}$  
\item $n$ even, $d_1$ and $d_2$ odd: $X$ has $10$ singular points and basket  $4\frac{k}{n}+2\frac{-k}{n/2}+2\frac{k}{n/2}+2\frac{-k}{n/2}$  
\item $n$ odd, $d_1$ and $d_2$ odd: $X$ has $9$ singular points and basket $5\frac{k}{n}+2\frac{k(n-1)/2}{n}+2\frac{-2k}{n}$  
\end{enumerate}

\fi

We consider the invariant $\gamma$ of the basket introduced in \cite{Product-Quotients}*{Section 4}: it vanishes by \cite{Product-Quotients}*{Proposition 4.4} since the basket contains as many points of type $\frac{k}{n}$ as of type $\frac{-k}{n}$. By  \cite{Product-Quotients}*{Proposition 4.1} $K^2_S=8\chi \left( {\mathcal O}_S \right) -2\gamma -l=8\chi \left( {\mathcal O}_S \right) -l$ where $l$ is the number of exceptional curves of $S \rightarrow X$.

Therefore
\[
8-\mu(S)=\frac{8\chi({\mathcal O}_S)-K^2_S}{\chi({\mathcal O}_S)}=
\frac{l}{\chi({\mathcal O}_S)}=\frac{l}{\frac{\left( n \frac{d_1}2 - 2\right)\left( n \frac{d_2}2 - 2\right)}n+4 \left( 1 - \frac1n\right)}
\]

Writing the continued function  of $\frac{n}{k}$
\begin{equation*}\label{contfrac}
\frac{n}{k}=b_1-\cfrac{1}{b_2-\cfrac{1}{b_3-\dots}} 
\end{equation*}

then (\cite{RiemenschneiderDef}*{Section 3}) the number of irreducible components of the resolution above two singular points of respective type $\frac{k}{n}$ and  $\frac{-k}{n}$ equals
$1+\sum (b_j-1)$, so 
\begin{equation}\label{8-mu}
8-\mu(S)=\frac{8\left( 1+\sum (b_j-1) \right)}{\frac{\left( n \frac{d_1}2 - 2\right)\left( n \frac{d_2}2 - 2\right)}n+4 \left( 1 - \frac1n\right)} \approx_{n \rightarrow \infty}\frac{32}{d_1d_2} \frac{ 1+\sum (b_j-1)}n 
\end{equation}

In the simplest case $k=1$ we obtain $\frac{ 1+\sum (b_j-1)}n =\frac{ 1+n-1}n=1$ and then
\begin{theorem}
There is an unbounded sequence $S_n$  of surfaces that have canonical map of degree $4$ such that
\[
\lim_{n \rightarrow \infty} \mu\left(S_n \right) = 8\left( 1-\frac{ 1}{m}\right)
\]  
for all positive integers $m \geq 6$ that are not prime numbers. 
\end{theorem}
\begin{proof}
Write $m=ab$ with $a\geq 2$, $b \geq 3$ and pick a sequence of  Wiman product-quotient surfaces $S_n$ of type $n,2a, 2b$ and shift $1$. 

We are in the assumptions of Theorem \ref{thm_deg4}, part (2)  ($d_1=2a \geq 4$, $d_2=2b \geq 6>5$) so the canonical map of  $S_n$ has degree $4$.

Finally, by \eqref{8-mu}
\[
\lim_{n \rightarrow \infty} \mu\left(S_n \right) = 8 - \frac{32}{d_1d_2}  \frac{ 1+\sum (b_j-1)}n =
8 - \frac{8}{ab} \frac{ 1+n-1}n = 8\left( 1-\frac{ 1}{ab}\right).
\]
\end{proof}

\section{Further questions and possible generalizations}
We have studied some natural generalizations of this construction giving surfaces with canonical map of degree $4$.
Unfortunately they do not lead to a substantial improvement of our main result, so  we have decided not to include them in this work.
However, we mention them here for completeness.

We obtain in fact similar results for Wiman product quotient surfaces where the $d_j$ are not both even. One can also consider hyperelliptic curves of equation $y^2=x_0x_1f(x_0^n,x_1^n)$. All these generalizations lead to surfaces with canonical map of degree $4$ and slope in the same range $\left[ 6 + \frac23, 8 \right]$.

The other possible generalization is by considering shifts other that $1$.
More precisely, consider a sequence of positive integers $k_n$, with $1 \leq k_n \leq n-1$, $\gcd(k_n,n)=1$.
Then a sequence $S_n$ of Wiman product-quotient surfaces of type $n,2a, 2b$ and shift $k_n$ has
\[
\lim_{n \rightarrow \infty} \mu(S_n) = 8-8\frac{1}{m} \lim_{n \rightarrow \infty} \sigma \left(\frac{k_n}{n} \right).
\]  
where
\[
\sigma \left( \frac{k}{n} \right):=\frac{ 1+\sum (b_j-1)}n.
\]

Obviously $\sigma \left( \frac{k}{n} \right)> 0$, $\sigma\left( \frac{1}{n} \right)=1$.
 It is known \cite{UrzuaTroncoso}*{Lemma 3.3} that $\sigma \leq 1$. An independent proof has been sent us by J. Stevens.

{\bf Question:} {\it What are the possible limits of $\left\{ \sigma \left( \frac{k}{n} \right) \right\} \subset[0,1]$ for sequences  of rational numbers $\frac{k}{n}$ with unbounded denominators?}

Note $\lim_{n \rightarrow \infty} \sigma \left( \frac{m}{mn+1}  \right)=\frac1{m}$.
We could not obtain any sequence with limit neither zero nor of the form $\frac1m$.  If there were more possible limits, this construction would improve our main result.

%=========================================================================

\end{document}